\documentclass[12pt,twoside,draft]{cmpart}
\usepackage[mathscr]{eucal}

\author{Serdyuk A.S., Shidlich A.L.}
\title{Direct and inverse theorems on the approximation of almost periodic functions in  Besicovitch-Stepanets spaces}

\shorttitle{Direct and inverse theorems on the approximation of almost periodic functions...}

\institute{Institute of Mathematics of the National Academy of Sciences of Ukraine, 3, Tereshchenkivska str., Kyiv, Ukraine, 01024 
}

\email{sanatolii@ukr.net (Serdyuk A.S.), shidlich@gmail.com (Shidlich A.L.)}

\enabstract{Serdyuk A.S., Shidlich A.L.}%
{Direct and inverse theorems on the approximation of almost periodic functions in  Besicovitch-Stepanets spaces}
{Direct and inverse approximation theorems are proved in the Besicovitch-Stepanets spaces $B{\mathcal S}^{p}$ of almost periodic functions in terms of the best approximations of functions and their generalized moduli of smoothness.}
{direct approximation theorem, inverse approximation theorem, Jackson type inequality, generalized module of smoothness}

\uaabstract{Сердюк А.С., Шидліч А.Л.}%
{Прямі та обернені теореми наближення майже періодичних функцій у просторах Безиковича-Степанця}
{Доведено прямі та обернені теореми наближення у просторах  Безиковича-Степанця $B{\mathcal S}^{p}$ майже періодичних функцій в термінах найкращих наближень функцій та їх уза\-галь\-нених модулів гладкості.}
{пряма апроксимаційна теорема, обернена апроксимаційна теорема, нерівність типу Джексона, узагальнений модуль гладкості}


\UDC{517.5}

\received{07.05.2021}  

\begin{document}

\maketitle


\section*{Introduction} In the paper, in terms of the best approximations and generalized moduli of  smoothness,
direct and inverse approximation theorems are established for Besicovitch almost periodic functions, the sequences of Fourier exponents of which have a single boundary point in infinity, and the  sums of  the $p$th degrees of absolute values of the Fourier coefficients are finite. Study of direct and inverse approximation theorems originates in the well-known papers of Jackson \cite{Jackson_1911}
and Bernstein \cite{Bernstein_1912}. Such theorems establish connections between the difference-differential properties of the function that is approximated and the value of the error of its approximation by various methods (see, for examples the monographs  \cite{A_Timan_M1960}, \cite{Butzer_Nessel_M1971},  \cite{Korneichuk_1987}, \cite{Stepanets_M2005}, etc.). In 1962 Korneichuk
\cite{Korneichuk_1962} proved Jackson inequality   in the uniform metric with the least (exact) constant.
In \cite{Chernykh_1967} Chernykh showed that for an arbitrary $2\pi$-periodic square-summable non-constant
function $f$ ($f\in L_2$),
  \begin{equation}\label{Chernykh_result}
E_n(f)_2<   2^{-1/2}\omega (f,\pi/n)_2, \quad n=1,2,\ldots \ (n\in {\mathbb N}),
\end{equation}
where $E_n(f)_2$ is the best mean-square approximation of the function $f$ by trigonometric polynomials of  order $n-1$,
and $\omega(f, t)_2 $ is the modulus of continuity (modulus of smoothness of the first order) of $f$ in the space $L_2$.
In  \cite{Chernykh_1967_MZ}, the unimprovable Jackson-type inequalities with averaged moduli of smoothness with some weight functions were established.

In spaces of almost periodic functions, direct approximation theorems were established in the papers \cite{Bredikhina_1968, Prytula_1972, Prytula_Yatsymirskyi_1983, Babenko_Savela_2012}, etc. In particular, in \cite{Prytula_1972}, an analogue of the inequality (\ref{Chernykh_result})  was obtained for  Besicovitch almost periodic functions of the order $2$ ($B_2$-a.p. functions).  In \cite{Prytula_Yatsymirskyi_1983} and \cite{Babenko_Savela_2012},   Jackson type inequalities  were obtained, respectively, with moduli of smoothness  of $B_2$-a.p. functions of arbitrary positive integer order
and with generalized moduli of smoothness.

In this paper, we consider the spaces $B {\mathcal S}^p$ of all functions that are  Besicovitch almost periodic of order $1$
($B$-a.p. Functions) for which  the sums of the $p$th degrees of absolute values of their  Fourier coefficients are finite,  $1\le p<\infty$.
The norm of a function  in the spaces $B {\mathcal S}^p$  is defined as the usual norm of a sequence of its Fourier coefficients in the space of numerical sequences $l_p$.

In the case when the $2\pi$-periodic Lebesgue summable functions were considered instead of the $B$-a.p. functions,
similar spaces were studied in the papers of Stepants and his followers, and they were denoted by  ${\mathcal S}^p$   \cite{Stepanets_2001, Stepanets_Serdyuk_2002, Voicexivskij_2002, Serdyuk_2003, Vakarchuk_2004}, \cite[Ch.~11]{Stepanets_M2005}, \cite{Abdullayev_Serdyuk_Shidlich_2021, Serdyuk_Shidlich_2021}, etc. In \cite{Stepanets_Serdyuk_2002},
direct and inverse theorems for the approximation of functions from the spaces ${\mathcal S}^p$
were proved in terms of their best approximations by trigonometric polynomials and moduli of smoothness of
 arbitrary positive  orders. In \cite{Abdullayev_Serdyuk_Shidlich_2021}, exact Jackson-type inequalities
  in the spaces  ${\mathcal S}^p$ were obtained in terms of the best approximations of functions and
  the averaged values of their generalized moduli of smoothness as well as the exact values  were found for widths
  of classes of $2\pi$-periodic functions defined by certain conditions on the averaged values of their generalized moduli of smoothness.

The spaces $B {\mathcal S}^p$ are a natural generalization of both the spaces ${\mathcal S}^p$ (since ${\mathcal S}^p\subset B {\mathcal S}^p$) and spaces $B_2$-a.p. functions (since the sets of $B_2$-a.p. functions coincide with the sets $B{\mathcal S}^2$). Therefore, it is of interest to obtain direct and inverse  theorems on the approximation of
functions from the spaces $B {\mathcal S}^p$ in terms of their best approximations and generalized moduli of smoothness.

\section{Preliminaries}
Let  $ {B}^s$,  $ 1\le s<\infty$, be the space of all functions  Lebesgue summable with the $s$th degrees in each finite interval of the real axis, in which the distance is defined by the equality
 \[
 D_{_{\scriptstyle  B^s}}(f,g)=\Big(\mathop{\overline{\lim}}\limits_{T\to \infty}\frac 1{2T}\int_{-T}^T
 |f(x)-g(x)|^s {\mathrm d}x\Big)^{1/s}.
 \]
 Further, let ${\mathscr T}$ be the set of all trigonometric sums of the form  $\tau_N(x)=\sum_{k=1}^N a_k {\mathrm e}^{{\mathrm i} \lambda_kx}$, $N\in {\mathbb N}$, where $\lambda_k$ and $a_k$ are arbitrary real and complex numbers ($\lambda_k\in {\mathbb R}$, $a_k\in {\mathbb C}$).

An arbitrary function $f$ is called a Besicovitch almost periodic function of order $s$ (or $B^s$-a.p. function) and is denoted by $f\in B^s$-a.p.  \cite[Ch.~5, \S10]{Levitan_M1953}, \cite[Ch.~2, \S7]{Besicovitch_M1955}, if there exists a sequence of trigonometric sums
$\tau_1, \tau_2, \ldots$ from the set  ${\mathscr T}$ such that
 \[
 \lim_{N\to \infty} D_{_{\scriptstyle  B^s}}(f,\tau_N)=0.
 \]

If $s_1\ge s_2\ge 1$, then (see, for example, \cite{Bredikhina_1968,  Bredikhina_1984}) $B^{s_1}$-a.p.$\subset B^{s_2}$-a.p.$\subset B$-a.p.,
where  $B$-a.p.$:=B^1$-a.p.   For any $B$-a.p. function $f$, there exists the average value
 \[
 M\{f\}=\lim\limits_{T\to \infty} \frac 1 T \int_0^T f(x){\mathrm d}x.
 \]
The value of the function $M\{f(\cdot) {\mathrm e}^{-{\mathrm i}\lambda \cdot} \}$, $\lambda\in {\mathbb R}$, can be nonzero at most on a countable set.  As a result of numbering the values of this set in an arbitrary order, we obtain a set ${\mathscr S}(f)=\{\lambda_k\}_{k\in {\mathbb N}}$  of Fourier exponents, which is called the spectrum of the function $f$.
The numbers  $A_{\lambda_k}=A_{\lambda_k}(f)=M\{f(\cdot) {\mathrm e}^{-{\mathrm i}\lambda_k \cdot} \}$ are called the Fourier coefficients of the function  $f$. To each function $f\in B$-a.p. with spectrum ${\mathscr S}(f)$  there corresponds a
Fourier series of the form
 $
 \sum_k A_{\lambda_k} {\mathrm e}^{ {\mathrm i}\lambda_k x}.
$
If, in addition, $f\in B_2$-a.p., then the Parseval equality holds (see, for example, \cite[Ch.~2, \S9]{Besicovitch_M1955})
 \[
 M\{|f|^2\}=\sum_{k\in {\mathbb N}} |A_{\lambda_k}|^2.
 \]

Developing the ideas of Stepanets \cite{Stepanets_2001}, for a fixed  $ 1\le p<\infty$ we consider the spaces of all functions
$f\in B$-a.p., for which the following quantity is finite:
 \begin{equation}\label{norm_Sp}
 \|f\|_{_{\scriptstyle  p}}:=
 \|f\|_{_{\scriptstyle  B{\mathcal S}^p}} =\|\{A_{\lambda_k}(f)\}_{k\in {\mathbb N}}\|_{_{\scriptstyle  l_p({\mathbb N})}} =
 \Big(\sum_{k\in {\mathbb N}}|A_{\lambda_k}(f)|^p\Big)^{1/p}.
\end{equation}
We denote these spaces by $B{\mathcal S}^p$  and call them Besicovitch-Stepanets spaces. By definition, $B$-a.p. functions are considered identical in $B{\mathcal S}^p$ if they have the same Fourier series.

Further, we will consider only those almost periodic functions from the spaces $B{\mathcal S}^p$, the sequences of Fourier exponents of which have a single limit point at infinity. For such functions $f$,  the Fourier series are written in
the symmetric form:
 \begin{equation}\label{Fourier_Series}
  S[f](x)=\sum _{k \in {\mathbb Z}}
  A_{k} {\mathrm e}^{ {\mathrm i}\lambda_k x},\quad  \mbox{\rm  where }\
  A_{k}=A_{k}(f)=M\{f(\cdot) {\mathrm e}^{-{\mathrm i}\lambda_k \cdot}\},
\end{equation}
$\lambda_0:=0$, $\lambda_{-k}=-\lambda_k$,  $|A_k|+|A_{-k}|>0$, $\lambda_{k+1}>\lambda_k>0$ при $k>0$.

By $G_{\lambda_n}$  we denote the set of all $B$-a.p. functions whose Fourier exponents belong to the interval
$(-\lambda_n,\lambda_n)$  and define the value of the best approximation by the equality
 \begin{equation}\label{Best_Approximation}
 E_{\lambda_n}(f)_{_{\scriptstyle p}} = E_{\lambda_n}(f)_{_{\scriptstyle  B{\mathcal S}^p}}  =\inf\limits_{g\in G_{\lambda_n}}
 \|f-g \|_{_{\scriptstyle p}}.
\end{equation}


Let $\Phi$ be the set  of all continuous nonnegative pair functions
$\varphi(t)$ such that $\varphi(0)=0$ and the Lebesgue measure of the set $\{t\in {\mathbb R}:\,\varphi(t)=0\}$
is equal to zero. For an arbitrary fixed $\varphi\in \Phi$, consider the generalized modulus of smoothness of the
function $f\in  B{\mathcal S}^p $
\begin{equation}\label{general_modulus}
    \omega_\varphi(f,\delta)_{_{\scriptstyle p}}
  :=\omega_\varphi(f,\delta)_{_{\scriptstyle   B{\mathcal S}^p}}  =\sup\limits_{|h|\le \delta} \Big(\sum_{k\in {\mathbb Z}}
\varphi^p(\lambda_kh)|  A_k(f) |^p\Big)^{1/p},\quad  \delta\ge 0.
\end{equation}

Let ${\mathcal M}=\{\mu_j\}_{j=0}^m$ be a nonzero  collection of complex numbers such that $\sum_{j=0}^m \mu_k=0$. We
associate the collection ${\mathcal M}$ with the difference operator $\Delta_h^{\mathcal M}f(t)=\sum_{j=0}^m \mu_j f(t-jh)$
and the modulus of smoothness
 \[
    \omega_{\mathcal M}(f,\delta)_{_{\scriptstyle p}} :=
    \sup\limits_{|h|\le \delta}\|\Delta_h^{\mathcal M} f\|_{_{\scriptstyle p}}.
 \]
Note that the collection ${\mathcal M}(m)=\Big\{\mu_j=(-1)^j {m \choose j}, \ j=0,1,\ldots, m\Big\}$, $m\in {\mathbb N}$, corresponds to the classical modulus of smoothness of order  $m$:  $\omega_{{\mathcal M}(m)}(f,\delta)_{_{\scriptstyle p}}=\omega_m(f,\delta)_{_{\scriptstyle p}}$.

For any  ${k}\in {\mathbb Z}$, the Fourier coefficients of the function $\Delta_h^Mf$ satisfy the equality
\[
|A_k(\Delta_h^{\mathcal M}f)|=|A_k(f)| \Big|\sum_{j=0}^m \mu_j {\mathrm e}^{-{\mathrm i}\lambda_kjh}\Big|.
\]
Therefore, for $\varphi_{\mathcal M}(t)=\Big|\sum_{j=0}^m \mu_j {\mathrm e}^{-{\mathrm i} jt}\Big|$ we have
 $
 \omega_{\varphi_{\mathcal M}}(f,\delta)_{_{\scriptstyle p}} =
 \omega_{\mathcal M}(f,\delta)_{_{\scriptstyle p}}.
 $
In particular, for  $\varphi_m(t)= 2^m   |\sin (t/2) |^m = 2^{\frac { m   }2} (1-\cos t)^{\frac m2}$, $m\in {\mathbb N}$, we have
$
 \omega_{\varphi_m}(f,\delta)_{_{\scriptstyle p}} =
 \omega_m(f,\delta)_{_{\scriptstyle p}},
 $

In the general case, such modules were studied  in
\cite{Boman_1980, Vasil'ev_2001, Kozko_Rozhdestvenskii_2004, Vakarchuk_2016, Babenko_Savela_2012, Babenko_Konareva_2019, Abdullayev_Chaichenko_Shidlich_2021, Abdullayev_Serdyuk_Shidlich_2021}, etc.



\section{Main results}

\subsection{Jackson type inequalities}
In this subsection, direct approximation theorems are established in the space $B{\mathcal S}^p$ in terms of best
approximations and generalized moduli of smoothness. For functions $f\in B{\mathcal S}^p$ with the Fourier series of the form (\ref{Fourier_Series}), we prove Jackson type inequalities of the kind as
 $$
 E_{\lambda_n}(f)_{_{\scriptstyle p}}\le
 K(\tau )\omega _\varphi\Big(f, \frac {\tau }{\lambda_n}\Big)_{_{\scriptstyle p}}, \quad \tau >0, \quad 1\le p<\infty, \quad n\in {\mathbb N},
 $$
and  consider the problem of the least constant in these inequalities for fixed values of
the parameters $n$, $\varphi$, $\tau$ and $p$. In particular, we study the quantity
 \[
 K_{n,\varphi,p}(\tau )=
  \sup \bigg \{\frac {E_{\lambda_n}(f)_{_{\scriptstyle p}}}{\omega _\varphi(f, \frac {\tau}{\lambda_n})_{_{\scriptstyle p}}}\ \ :
  f\in B{\mathcal S}^p\bigg\}.
 \]
Here and below, we assume that $0/0=0$.

Let $V(\tau)$, $\tau>0$, be a set of bounded nondecreasing functions $v $ that differ from a constant on $[0, \tau].$

 \begin{theorem}
   \label{Th.1}
   Assume that the function $f\in B{\mathcal S}^p,$ $1\le p<\infty$, has the Fourier series of the form (\ref{Fourier_Series}). Then for any  $\tau >0$, $n\in {\mathbb N}$ and $\varphi\in \Phi$ the following inequality holds:
\begin{equation}\label{(6.7)}
 E_{\lambda_n}(f)_{_{\scriptstyle p}}\le C_{n,\varphi,p}(\tau ) \omega _\varphi\Big(f, \frac {\tau }{\lambda_n}\Big)_{_{\scriptstyle p}},
  \end{equation}
where
 \begin{equation}\label{(6.8)}
 C_{n,\varphi,p}(\tau ):=\left(\inf\limits _{v  \in  M(\tau )} \frac {v  (\tau )-v  (0)}
 {I_{n,\varphi,p}(\tau ,v  )}\right)^{1/p},
  \end{equation}
 and
 \begin{equation}\label{(6.9)}
  I_{n,\varphi,p}(\tau ,v  ):= \inf\limits _{{k \in {\mathbb N}},\,k \ge n}  \int\limits _0^{\tau }\varphi^p\Big(\frac{\lambda_k t}{\lambda_n} \Big){\mathrm d}v  (t).
 \end{equation}
 Futhermore, there exists a function $v _*\in V(\tau )$ that realizes the greatest lower bound in (\ref{(6.8)}). Inequality (\ref{(6.7)}) is
unimprovable on the set of all functions $f\in B{\mathcal S}^p$
 with the Fourier series of the form (\ref{Fourier_Series})  in the sense that for any $\varphi\in \Phi$ and $n\in {\mathbb N}$
 the following equality is true:
 \[
 C_{n,\varphi,p}(\tau )= K_{n,\varphi,p}(\tau ).
 \]
\end{theorem}

In the spaces  $L_2$ of $2\pi$-periodic square-summable
functions, for moduli of continuity, this result was obtained by Babenko \cite{Babenko_1986}.
In the spaces ${\mathcal S}^p$  of functions of one and several variables, this result for classical moduli of smoothness was obtained, respectively, \cite{Stepanets_Serdyuk_2002} and   \cite{Abdullayev_Ozkartepe_Savchuk_Shidlich_2019}, and for generalized moduli of smoothness, in \cite{Abdullayev_Chaichenko_Shidlich_2021} (for functions of one variable).
In the proof of Theorem \ref{Th.1}, we mainly use the ideas outlined in
\cite{Stepanets_Serdyuk_2002, Babenko_1986,
Chernykh_1967, Chernykh_1967_MZ},   taking into account the peculiarities of the spaces $B{\mathcal S}^p$.

\begin{proof}

From relations (\ref{norm_Sp}) and (\ref{Best_Approx}), it follows that for any   $f\in B{\mathcal S}^p$
with  the Fourier series of the form (\ref{Fourier_Series}), we have
  \begin{equation} \label{Best_Approx}
     E_{\lambda_n}^p(f)_{_{\scriptstyle p}}=\|f-S_n(f)\|_{_{\scriptstyle p}} =\sum\limits_{|k|\ge n} |A_k(f)|^p,
 \end{equation}
where $S_n(f):=\sum _{|k|<n} A_{k}(f) {\mathrm e}^{ {\mathrm i}\lambda_k x}$.

For any  $f\in B{\mathcal S}^p$, $\varphi\in \Phi$ and $h\in {\mathbb R}$, consider the sequence  of numbers
$ \{\varphi(\lambda_kh) A_k(f) \}_{k\in {\mathbb Z}}$. If there exists a function
$\Delta_h^\varphi f\in B$-a.p. such that for all $k\in {\mathbb Z}$
  \begin{equation}
     \label{Fourier_Coeff_Delta}
     A_k(\Delta_h^\varphi f)=\varphi(\lambda_kh) A_k(f),
 \end{equation}
then denote by  $\|\Delta_h^\varphi f\|_{_{\scriptstyle p}}$
the usual norm  (\ref{norm_Sp}) of the function  $\Delta_h^\varphi f$. If such a $B$-a.p function $\Delta_h^\varphi f$
does not exist, then to simplify notation we will also use the notation $\|\Delta_h^\varphi f\|_{_{\scriptstyle p}}$,
meaning by it the $l_p$-norm of the sequence $ \{\varphi(\lambda_kh) A_k(f) \}_{k\in {\mathbb Z}}$.

 Taking into account (\ref{Best_Approx}),  (\ref{(6.9)})) and the parity of the function $\varphi$, we obtain
 \[
  \|\Delta_h^\varphi f\|_{_{\scriptstyle p}}^p=
  \sum_{{k}\in {\mathbb Z}} \varphi^p(\lambda_kh)|A_k(f)|^p
\ge  \sum_{|k|\ge n} \varphi^p(\lambda_kh) |A_k(f)|^p
\]
 \[
 =\frac {I_{n,\varphi,p}(\tau ,v  )}{v  (\tau )-v  (0)}E_{\lambda_n}^p(f)_{_{\scriptstyle p}}+
 \sum_{|k|\ge n}   |A_k(f)|^p\Big(\varphi^p(\lambda_kh)-\frac {I_{n,\varphi,p}(\tau ,v  )} {v  (\tau ) - v  (0)}\Big),
 \]
where the quanity $I_{n,\varphi,p}(\tau ,v  )$ is defined by  (\ref{(6.9)}). Hence, for any $ t\in [0,\tau]$ we find
  \begin{equation}\label{(6.32)}
 E_{\lambda_n}^p(f)_{_{\scriptstyle p}}\le
 \frac {v  (\tau )-v  (0)} {I_{n,\varphi,p}(\tau ,v  )}
 \Bigg( \|\Delta_{t/{\lambda_n}}^\varphi f\|_{_{\scriptstyle p}}^p             -
    \sum_{|k|\ge n}   |A_k(f)|^p\bigg(\varphi^p\Big(\frac{\lambda_k t}{\lambda_n}\Big)-\frac {I_{n,\varphi,p}(\tau ,v  )}{v
 (\tau )-v  (0)}\bigg) \Bigg).
 \end{equation}
 Since the both sides of inequality (\ref{(6.32)}) are nonnegative and the series on its right-hand side is majorized on
the entire real axis by the absolutely convergent series
 $
 C(\varphi)\sum_{|k|\ge n}   |A_k(f)|^p,
 $
where $C(\varphi) =\max\limits_{t\in {\mathbb R}}\varphi(t)$, then integrating this inequality
with respect to ${\mathrm d}v  (t)$  from  $0$ to $\tau,$ we get
 \[
 E_{\lambda_n}^p(f)_{_{\scriptstyle p}}(v  (\tau )-v
 (0))\le \frac {v  (\tau )-v  (0)} {I_{n,\varphi,p}(\tau ,v  )}
 \Bigg(\int\limits _0^{\tau } \|\Delta_{t/{\lambda_n}}^\varphi f\|_{_{\scriptstyle p}}^p  {\mathrm d}v  (t)
 \]
 \[
 - \sum_{|k|\ge n}   |A_k(f)|^p\bigg(\int\limits  _0^{\tau }\varphi^p\Big(\frac{\lambda_k t}{\lambda_n}\Big){\mathrm d}v  (t)-I_{n,\varphi,p}(\tau ,v  )\bigg)\Bigg).
 \]
By  virtue of (\ref{(6.9)}), we have маємо
 \[
 \int\limits  _0^{\tau }\varphi^p\Big(\frac{\lambda_k t}{\lambda_n}\Big){\mathrm d}v  (t)-
 I_{n,\varphi,p}(\tau ,v  )\ge 0.
 \]
Therefore, for any function $v  \in  V(\tau )$, we have
  \begin{equation}\label{(6.34)}
 E_{\lambda_n}^p(f)_{_{\scriptstyle p}}\le \frac 1{I_{n,\varphi,p}(\tau ,v  )}\int\limits _0^{\tau}
 \|\Delta_{t/{\lambda_n}}^\varphi f\|_{_{\scriptstyle p}}^p  {\mathrm d}v  (t)
 \le \frac 1{I_{n,\varphi,p}(\tau ,v  )}\int\limits _0^{\tau }\omega _\varphi^p\Big(f, \frac t{\lambda_n}\Big){\mathrm d}v  (t).
 \end{equation}
Hence we immediately get (\ref{(6.7)}) and the estimate
\begin{equation}\label{(6.35)}
  K_{n,\varphi,p}^p(\tau )\le \inf\limits _{v  \in V(\tau )}\frac {v  (\tau
  )-v  (0)}{I_{n,\varphi,p}(\tau ,v  )}=C_{n,\varphi,p}^p(\tau ).
 \end{equation}
 It remains to show that  in relation (\ref{(6.35)})  there is in fact equality

By virtue of (\ref{general_modulus}) and (\ref{Best_Approx}), we have
 \begin{equation}\label{K_n_new}
   K_{n,\varphi,p}^p(\tau )
   = \sup\limits_{f\in B{\mathcal S}^p }
   \frac { \sum _{|k| \ge n}  |A_k(f)|^p}
    {\sup _{|h|\le \tau}\sum _{|k| \ge n}  \varphi^p(\lambda_k h/\lambda_n) |A_k(f)|^p}.
 \end{equation}
In (\ref{K_n_new}), it is sufficient to consider the  supremum over all functions $f\in B{\mathcal S}^p$, such that
 $\sum _{|k| \ge n}  |A_k(f)|^p\le 1$. Then, taking into account the parity of the function $\varphi$, we obtain
 \begin{equation}\label{K_n_new1}
   K_{n,\varphi,p}^{-p}(\tau)  \le
   J_{n,\varphi,p}(\tau):=\inf\limits_{w\in W_{n,\varphi,p}}\|w\|_{_{\scriptstyle  C_{[0,\tau]}}},
 \end{equation}
where the set
\begin{equation}\label{W_{n,varphi}}
 W_{n,\varphi,p}:=\bigg\{\omega(u)=\sum_{j=n}^\infty \varrho_j \varphi^p\Big(\frac{\lambda_ju}{\lambda_n}\Big): \varrho_j\ge 0, \ \sum_{j=n}^\infty \varrho_j=1 \bigg\}.
 \end{equation}
Further, we use the duality relation in the space $C_{[a,b]},$ which we formulate as the following statement  (see, e.g., \cite[Ch.~1.4]{Korneichuk_1987}):

\begin{proposition}\label{Prop1}{\rm \cite[Ch.~1.4]{Korneichuk_1987}}.
        If $F$ is a convex set in the space $C_{[a,b]},$ then for any $x \in C_{[a,b]}$
        \begin{equation}\label{(6.38)}
        \inf\limits_{u\in F}\|x-u\|_{_{C_{[a,b]}}}=\sup\limits _{{{\mathop {V}\limits_a^b}}(g)\le 1}\Big( \int\limits _a^bx(t){\mathrm d}g(t)-\sup\limits  _{u\in F}\int\limits _a^bu(t){\mathrm d}g(t)\Big).
        \end{equation}
 For $x\in C_{[a,b]}\setminus \bar F$, where $\bar F$ is the closure of a set $F$, there exists a function $g$ with variation equal to 1 on  $[a,b]$ that realizes the least upper bound in (\ref{(6.38)}).
\end{proposition}

It is easy to see that the set $W_{n,\alpha,p}$ is a convex subset of the space $C_{[0, \tau]}$. Therefore, setting $a=0,$ $b=\tau,$ $x(t)\equiv 0,$ $u(t)=w(t) \in W_{n,\alpha,p},$ $F=W_{n,\alpha,p},$ from relation (\ref{(6.38)}) we get
\begin{equation}\label{(6.39)}
   J_{n,\varphi,p}(\tau )=\inf\limits _{w\in W_{n,\varphi,p}}\|0-w\|_{_{C_{[0, \tau ]}}}
   $$
   $$
   =\sup\limits _{\mathop {V}\limits _0^{\tau}(g)\le 1}\bigg(0
   -\sup\limits _{w\in W_{n,\varphi,p}}\int\limits _0^{\tau}w(t){\mathrm d}g(t)\bigg)= \sup\limits _{\mathop {V}\limits _0^{\tau }(g)
   \le 1}\inf\limits _{w\in W_{n,\varphi,p}}\int\limits _0^{\tau } w(t){\mathrm d}g(t).
 \end{equation}
Furthermore, according to Proposition \ref{Prop1}, there exists a function $g_*(t),$; that realizes the least upper bound in
(\ref{(6.39)}) and such that $\mathop {V}\limits _0^{\tau }(g_*)=1$.  Every function $w\in
W_{n,\alpha,p}$ is nonnegative. Therefore, it sufficient to take the supremum
on the right-hand side of (\ref{(6.39)}) over the set of nondecreasing functions $v  (t)$ for which $v  (\tau )-v  (0)\le 1$. For such functions, by virtue of (\ref{(6.9)}) and (\ref{W_{n,varphi}}),
the following equality is true:
 \[
 \inf\limits _{w\in W_{n,\alpha,p}}\int\limits _0^{\tau }w(t){\mathrm d}v  (t)=I_{n,\varphi,p}(\tau ,v   ).
 \]
 This implies that there exists a function   $v _*\in V(\tau )$ such that  $v  _*(\tau )-v  _*(0)=1$ and
  \begin{equation}\label{(6.41)}
  I_{n,\varphi,p}(\tau ,v  _*)= \sup\limits _{v  \in V(\tau ),\, \mathop {V}\limits _0^{\tau }(v  )\le 1}  I_{n,\varphi,p}(\tau ,v   )=
  J_{n,\varphi,p}(\tau ).
  \end{equation}
From relations (\ref{K_n_new1}) and (\ref{(6.41)}), we obtain the necessary estimate:
    \[
     K_{n,\varphi,p}^p(\tau )\ge \frac 1{
  J_{n,\varphi,p}(\tau )
  }
  =\frac 1{ I_{n,\varphi,p}(\tau ,v  _*)}=\frac {v _*(\tau )-v _*(0)} {I_{n,\varphi,p}(\tau ,v _*)}=C_{n,\varphi,p}^p(\tau ).
  \]

 \end{proof}

Consider an important special case when
$\varphi(t)=\varphi_\alpha(t)=  2^{\frac {\alpha}2} (1-\cos t)^{\frac \alpha2}=2^\alpha   |\sin (t/2) |^\alpha $, $\alpha>0$.
In this case, we set $ \omega_{\varphi_\alpha}(f,\delta)_{_{\scriptstyle p}}=: \omega_\alpha(f,\delta)_{_{\scriptstyle p}}$
and  $K_{n,\varphi_\alpha,p}(\tau )=:K_{n,\alpha,p}(\tau )$.
For the weight function $v  _1(t) = 1 - \cos t$, we get the following assertion:

 \begin{corollary}
   \label{Th.2} For any function $f\in B{\mathcal S}^p,$ $1\le p<\infty$,
   with the Fourier series of the form  (\ref{Fourier_Series}), the following inequalities holds:
 \begin{equation}\label{(6.11)}
 E_{\lambda_n}^p (f)_{_{\scriptstyle  {p}}} \le \frac 1{2^{\frac {\alpha  p }2}I_n(\frac {\alpha p }2)}\int\limits_0^{\pi }
 \omega _\alpha^p\Big(f, \frac t{\lambda_n}\Big)_{_{\scriptstyle  {p}}}  \sin t\,{\mathrm d}t, \ \ n \in {\mathbb N},\ \alpha>0,
  \end{equation}
 where
 \begin{equation}\label{(6.12)}
 I_n(s):=\inf\limits _{{k \in {\mathbb N}},\,k \ge n} \int\limits _0^{\pi } \Big(1-\cos \frac {\lambda_k t}{\lambda_n}\Big)^{s}\sin t\, {\mathrm d}t, \ \
 s >0, \ \ n\in {\mathbb N}.
   \end{equation}
If, in addition $\frac {\alpha p }2\in {\mathbb N}$, then
 \begin{equation}\label{(6.13)}
 I_n\Big(\frac { \alpha p }2\Big)=\frac {2^{\frac { \alpha  p }2+1}}{\frac { \alpha  p }2+1},
   \end{equation}
and inequality (\ref{(6.11)}) cannot be improved for any  $n\in
{\mathbb N}.$
 \end{corollary}

\begin{proof} Inequality (\ref{(6.11)}) follows from relation  (\ref{(6.34)}) with $\tau =\pi $, $\varphi(t)=\varphi_\alpha(t)$   and $v  (t)=1-\cos t,$ $t\in [0, \pi].$
In \cite[relation (52)]{Stepanets_Serdyuk_2002}, it was shown that for any  $\theta\ge 1$ and $s\in {\mathbb N}$
the following inequality holds:
 \[
 \int_0^\pi (1-\cos\theta t)^s \sin t {\mathrm d}t\ge  \frac {2^{s+1}}{s+1}.
 \]
which turns into equality for $\theta=1$. Therefore, setting  $s=\frac {\alpha p }2$ and $\theta=\frac {\lambda_\nu}{\lambda_n}$, $\nu=n,n+1,\ldots$,
and  the monotonicity  of the sequence of Fourier exponents $\{\lambda_k\}_{k\in {\mathbb Z}}$,
we see that for $\frac { \alpha p }2\in {\mathbb N}$,
indeed, the equality (\ref{(6.13)}) holds.

To prove that inequality (\ref{(6.11)}) is unimprovable for $\frac {\alpha  p
}2\in {\mathbb N}$, it suffices to verify that  the function
  \begin{equation}\label{A6.53}
 f^*(x)=\gamma +\beta e^{-\lambda_n x} + \delta e^{\lambda_n x},
     \end{equation}
where $\gamma$, $\beta$ and $\delta $ are arbitrary complex numbers, satisfies the equality
  \begin{equation}\label{(6.53)}
  E_{\lambda_n}^p (f^*)_{_{\scriptstyle  {p}}}  =\frac {\frac { \alpha  p }2+1}{2^{ \alpha  p+1}}\int\limits _0^{\pi }
  \omega _ \alpha ^p\Big(f^*, \frac t{\lambda_n}\Big)_{_{\scriptstyle  {p}}}   \sin t\, {\mathrm d}t, \ \  \ n\in {\mathbb N},\ \  \alpha >0.
     \end{equation}
In this case,  $E_{\lambda_n}^p (f^*)^p_{_{\scriptstyle  {p}}} =|\beta |^p+|\delta |^p$, the function
  $
    \|\Delta_{{t}/{\lambda_n}}^ {\varphi_\alpha}  f^*\|_{_{\scriptstyle  {p}}} ^p
    =     2^{\frac { \alpha  p }2}(|\beta |^p+|\delta |^p) (1-\cos t)^{\frac { \alpha  p }2}
 $
does not decrease with respect to $t$ on $[0,\pi ].$ Therefore,
 $
 \omega _ \alpha  (f^*, \frac t{\lambda_n} )_{_{\scriptstyle  {p}}} =
 \|\Delta_{t/\lambda_n}^ {\varphi_\alpha}  f^*\|_{_{\scriptstyle  {p}}},
 $
and
 \[
 \frac {2^{ \alpha  p+1}} {\frac { \alpha  p }2+1}E_{\lambda_n}^p (f^*)_{_{\scriptstyle  {p}}} -\int\limits _0^{\pi }
  \omega _ \alpha ^p\Big(f^*, \frac t{\lambda_n}\Big)_{_{\scriptstyle  {p}}}   \sin t\, {\mathrm d}t
  \]
  \[
  =(|\beta |^p+|\delta |^p)\Big(\frac {2^{ \alpha  p+1}} {\frac { \alpha  p }2+1}-2^{\frac { \alpha  p }2}
  \int_0^\pi (1-\cos  t)^{\frac { \alpha  p }2} \sin t {\mathrm d}t\Big)=0
  \]
\end{proof}

If the weight function $v  _2(t)=t$, then we obtain the following assertion:
 \begin{corollary}
   \label{Th.2a} Assume that the function $f\in B{\mathcal S}^p,$ $1\le p<\infty$,
   has the Fourier series of the form (\ref{Fourier_Series}) and the number
   $\alpha>0$ such that  $\alpha p\ge 1$. Then for any  $0<\tau \le \frac {3\pi }4$ and $n \in {\mathbb N}$,
  \begin{equation}\label{(A6.53)}
    E_{\lambda_n}^p (f)_{_{\scriptstyle  {p}}} \le \frac {1}{2^{\alpha p}\int _0^{\tau }\sin ^{\alpha p }\frac t{2}{\mathrm d}t} \int\limits_0^{\tau}
 \omega _\alpha^p\Big(f, \frac t{\lambda_n}\Big)_{_{\scriptstyle  {p}}} {\mathrm d}t .
      \end{equation}
Equality in  (\ref{(A6.53)}) holds for the function $f^*$  of the form (\ref{A6.53}).
 \end{corollary}

\begin{proof} From inequality (\ref{(6.34)}), it follows that
 \[
     E_{\lambda_n}^p(f)_{_{\scriptstyle p}}
      \le \frac 1{\tilde{I}_n(\frac {\alpha p }2)}\int\limits _0^{\tau }
      \omega _\alpha^p\Big(f, \frac t{\lambda_n}\Big){\mathrm d}t,
 \]
where
 \[
 \tilde{I}_n(s):=\inf\limits _{{k \in {\mathbb N}},\,k \ge n} \int\limits _0^{\tau } \Big(1-\cos \frac {\lambda_k t}{\lambda_n}\Big)^{s} {\mathrm d}t, \ \
 s >0, \ \ n\in {\mathbb N}.
 \]
Consider the function
 \[
 F_\beta(x):=\frac 1x \int_0^x |\sin t|^\beta {\mathrm d}t.
 \]
In \cite{Voicexivskij_2002}, it is shown that for any  $h\in (0,\frac{3\pi}4)$ and $\beta\ge 1$, the following relation is true:
   \begin{equation}\label{W1}
   \inf\limits_{x\ge h/2} F_\beta(x)=F_\beta(h/2).
      \end{equation}
Since for any  $\theta\in {\mathbb R}$,
 \[
 \int\limits _0^{\tau } (1-\cos \theta t)^{\beta} {\mathrm d}t=
 \int\limits _0^{\tau }\Big|\sin \frac {\theta t}{2}\Big|^{\beta}   {\mathrm d}t=
  \tau F_{\alpha p}\Big( {\frac {\theta \tau}{2}}\Big),
 \]
then setting  $\theta=\frac{\lambda_k}{\lambda_n}\ge 1$ ($k\ge n$) and $\beta=\alpha p$, from (\ref{W1})
(with $\tau \in (0, \frac {3\pi }4]$) we obtain  
 \[
      \tilde{I}_n\Big(\frac {\alpha p }2\Big)=
          \inf\limits _{{k \in {\mathbb N}}:k  \ge n}
                    \int\limits _0^{\tau }\Big(1 - \cos\frac {\lambda_k t}{\lambda_n}\Big)^{\frac {\alpha p }2} {\mathrm d}t
                    = \inf\limits _{{k \in {\mathbb N}}:k  \ge n}
                    \int\limits _0^{\tau }\Big|\sin \frac {\lambda_k t}{2\lambda_n}\Big|^{\alpha p}   {\mathrm d}t = \int\limits _0^{\tau }  \sin^{\alpha p} \frac {t}{2}   {\mathrm d}t
                   .
 \]
For the functions  $f^*$  of the form (\ref{A6.53}), the equality
 \[
  E_{\lambda_n}^p (f^*)_{_{\scriptstyle  {p}}} = \frac {1}{2^{\alpha p}\int _0^{\tau }\sin ^{\alpha p }\frac t{2}{\mathrm d}t} \int\limits_0^{\tau}
 \omega _\alpha^p\Big(f^*, \frac t{\lambda_n}\Big)_{_{\scriptstyle  {p}}} {\mathrm d}t .
 \]
is verified similarly to the proof of equality (\ref{(6.53)}).
\end{proof}

In the following assertion, we give the upper estimates for the least constants $K_{n, \alpha ,p}(\tau ) $ in Jackson type inequalities
with the moduli of smoothness  $\omega _\alpha(f, \cdot)_{_{\scriptstyle  {p}}}$ and $\tau=\pi$. These estimates do not depend on
$n$ and are unimprovable in several important cases.

 \begin{corollary}
   \label{Th.3} For any $n\in {\mathbb N}$ and $\alpha>0$, the following inequalities are true:
  \begin{equation}\label{(6.14)}
 K_{n, \alpha ,p}^p(\pi )\le \frac 1{2^{\frac { \alpha  p }2-1}I_n(\frac { \alpha  p }2)}\le \frac {\frac { \alpha  p }2+1}{2^{ \alpha  p}+2^{\frac { \alpha  p }2-1}(\frac { \alpha  p}2+1)\sigma (\frac { \alpha  p }2)}, \end{equation}
 where the quantities  $I_n(s)$, $ s >0$, are defined by (\ref{(6.12)}), and
 \[
 \sigma ( s ):=-\sum _{\alpha=[\frac { s }2]+1}^{\infty } { s \choose 2\alpha}
\frac 1{2^{2\alpha-1}}\bigg(\frac {1-(-1)^{[ s ]}}2 {2\alpha \choose \alpha}
 - \sum _{j=0}^{\alpha-1} {2\alpha \choose j}\frac 2{2(\alpha-j)^2-1}\bigg),
 \]
($[ s ]$ is the integer part of the number $ s $). If  $\frac { \alpha  p }2\in {\mathbb N},$
then the value  $\sigma (\frac { \alpha  p }2)=0$ and
 $$
  K_{n, \alpha ,p}^p(\pi )\le \frac {\frac { \alpha  p }2+1}{2^{ \alpha  p}}, \ \ \
\frac { \alpha  p }2\in {\mathbb N}, \ \ \ n\in {\mathbb N}.
 \eqno (\ref{(6.14)}')
 $$
  \end{corollary}

\begin{proof}
The first inequality in (\ref{(6.14)}) and inequality $(\ref{(6.14)}')$ follow from Corollary \ref{Th.2}. The second inequality in
(\ref{(6.14)}) follows from the relation
 \[
 I_n\Big(\frac { \alpha  p }2\Big)\ge \frac {2^{\frac { \alpha  p
 }2+1}}{\frac { \alpha  p }2+1} + \sigma \Big(\frac { \alpha  p }2\Big), \ \ n\in {\mathbb N},\ \  \alpha >0,
 \]
 which is a consequence of the   inequality \cite{Stepanets_Serdyuk_2002}:
\[
 \int\limits _0^{\pi }(1-\cos \theta t)^{s }\sin tdt\ge \frac {2^{s +1}}{s +1} + \sigma (s ), \quad \theta \ge 1, \quad \ s >0.
\]
\end{proof}

The following assertion  establishes the uniform boundedness of the constants $K_{n, \alpha ,p}(\pi )$
with respect to the parameters  $n\in {\mathbb N}$ and $1\le p<\infty$.

\begin{corollary} \label{Theorem 2.4.}
   Assume that the function  $f\in B{\mathcal S}^p,$ $1\le p<\infty$, has the Fourier series of the form (\ref{Fourier_Series})  and  $\|f-A_0(f)\|_{_{\scriptstyle  p}}\not= 0$. Then for any   $n\in {\mathbb N}$ and $\alpha>0$,
  \begin{equation}\label{(6.16)}
       E_{\lambda_n} (f)_{_{\scriptstyle  {p}}} < \frac {(4/3)^{1/p}}{2^{\alpha/2}}\omega_\alpha
  \Big(f, \frac {\pi }{\lambda_n}\Big)_{_{\scriptstyle p}} \le  \frac {4}{3\cdot 2^{\alpha/2}}\omega_\alpha
  \Big(f, \frac {\pi }{\lambda_n}\Big)_{_{\scriptstyle p}}.
  \end{equation}
  Furthermore, in the case where $\alpha=m\in {\mathbb N}$, the following more accurate estimate holds:
  \begin{equation}\label{(6.161)}
       E_{\lambda_n} (f)_{_{\scriptstyle  {p}}} < \frac {4-2\sqrt{2}}{2^{m/2}}\omega_m
  \Big(f, \frac {\pi }{\lambda_n}\Big)_{_{\scriptstyle p}}.
   \end{equation}
\end{corollary}

\begin{proof} It was shown in   \cite{Stepanets_Serdyuk_2002} that  $I_n(s)\ge 2$ when $s\ge 1$
and  $I_n(s)\ge 1+2^{s-1}$ when $s\in (0,1)$. Combining these two estimates and (\ref{(6.14)}), we get   (\ref{(6.16)}).

Relation (\ref{(6.161)}) follows from the estimate $I_n(\frac{mp}2)\ge 1+\frac 1{\sqrt{2}}$, which is a consequence of the above estimates for the value of $I_n(s)$
in the case when $m\in {\mathbb N}$ and $1\le p<\infty$ \cite{Stepanets_Serdyuk_2002}.
 \end{proof}


As noted above, for $p=2$ the sets $B{\mathcal S}^p=B{\mathcal S}^2$ coincide with the sets of $B_2$-a.p. functions.
Given the importance of this case, we give the formulation of the corresponding statements for the classical modulus of smoothness
$\omega_m $, which follow from the consequence \ref{Th.2}.

 \begin{corollary} \label{Theorem 2.5.}
     For any $B_2$-a.p. function $f$ having the Fourier series of the form (\ref{Fourier_Series}),
  \begin{equation}\label{(6.162)}
     E_{\lambda_n}^2 (f)_{_{\scriptstyle  {2}}} \le \frac {m+1}{2^{2m+1}}
     \int\limits_0^{\pi } \omega _m^2\Big(f, \frac t{\lambda_n}\Big)_{_{\scriptstyle  {2}}}  \sin t\,{\mathrm d}t,\quad m,\, n\in {\mathbb N}.
   \end{equation}
Inequalities   (\ref{(6.162)}) can not be improved for any  $m$ and $n\in {\mathbb N}$.
\end{corollary}

 \begin{corollary} \label{Theorem 2.6.}
     For any $B_2$-a.p. function $f$ with the Fourier series of the form (\ref{Fourier_Series}) such that
     $\|f-A_0(f)\|_{_{\scriptstyle  p}}\not= 0$,  the following inequalities hold:
  \begin{equation}\label{(6.163)}
       E_{\lambda_n} (f)_{_{\scriptstyle  {2}}} < \frac {\sqrt{m+1}}{2^m}\omega_m
  \Big(f, \frac {\pi }{\lambda_n}\Big)_{_{\scriptstyle 2}},\quad m,\, n\in {\mathbb N}.
   \end{equation}
\end{corollary}


Inequalities (\ref{(6.162)}) and (\ref{(6.163)}) complement the results obtained in \cite{Babenko_Konareva_2019, Prytula_1972, Prytula_Yatsymirskyi_1983}, etc.,
for the $B_2$-a.p. functions. In the spaces  ${\mathcal S}^p$  of functions of one and several variables, Theorem  \ref{Th.1} and
Corollaries  \ref{Th.2}, \ref{Th.3} and \ref{Theorem 2.4.} were proved in \cite{Stepanets_Serdyuk_2002} and
 \cite{Abdullayev_Ozkartepe_Savchuk_Shidlich_2019}, respectively. In the spaces $L_2$,
 for classical moduli of smoothness inequality  (\ref{(6.11)}) was proved by Chernykh \cite{Chernykh_1967_MZ}.
 The inequalities of this type were also investigated in   \cite{Taikov_1976, Taikov_1979, Vasil'ev_2001, Voicexivskij_2002, Serdyuk_2003, Vakarchuk_2004, Gorbachuk_Grushka_Torba_2005, Babenko_Savela_2012, Vakarchuk_2016, Babenko_Konareva_2019, Abdullayev_Serdyuk_Shidlich_2021}, etc.



\section{Inverse approximation theorems.}

 \begin{theorem}
       \label{Inverse_Theorem}
       Assume that the function   $ f\in  B{\mathcal S}^p,$ $1\le p<\infty$,
         has the Fourier series of the form (\ref{Fourier_Series}), the function
         $\varphi\in \Phi$ does not decrease  on  $[0,\tau]$, $\tau>0$, and
       $\varphi(\tau)=\max\{\varphi(t):t\in {\mathbb R}\}$. Then for any  $n\in {\mathbb N}$,
       the following inequality holds:
       \begin{equation}\label{S_M.12}
       \omega _\varphi^p\Big(f, \frac{\tau}{\lambda_n}\Big)
       \le    \sum _{\nu =1}^{n}\Big(\varphi^p\Big(\frac{\tau\lambda_\nu }{\lambda_n}\Big)-
          \varphi^p\Big(\frac{ \tau\lambda_{\nu-1}}{\lambda_n}\Big)\Big) E_{\lambda_\nu}^p(f)_{_{\scriptstyle p}}
       .
       \end{equation}
 \end{theorem}

\begin{proof}
Let us use the scheme of the proof from  \cite{Stepanets_Serdyuk_2002, Abdullayev_Chaichenko_Shidlich_2021},
 taking into account the peculiarities of the spaces $B{\mathcal S}^p$.
As above, for any $f\in B{\mathcal S}^p$, $\varphi\in \Phi$ and $h\in {\mathbb R}$, we denote by  $\|\Delta_h^\varphi f\|_{_{\scriptstyle p}}$ the usual norm (\ref{norm_Sp}) of the function $\Delta_h^\varphi f$ satisfying relation (\ref{Fourier_Coeff_Delta}) (if such a function $\Delta_h^\varphi f\in B$-a.p. exists) or the  $l_p$-norm of the sequence  $ \{\varphi(\lambda_kh) A_k(f) \}_{k\in {\mathbb Z}}$ if such a function $\Delta_h^\varphi f\in B$-a.p. does not exist.
We have
\begin{equation}\label{Proof_Th2_1}
   \|\Delta_h^\varphi f\|_{_{\scriptstyle p}}^p=
  \sum_{{k}\in {\mathbb Z}} \varphi^p(\lambda_kh)|A_k(f)|^p
  =\sum_{|k|< n} \varphi^p(\lambda_kh) |A_k(f)|^p+\sum_{|k|\ge n} \varphi^p(\lambda_kh) |A_k(f)|^p.
   \end{equation}
 It is clear that the second term on the right-hand side (\ref{Proof_Th2_1})  does not exceed the value
 \[
 \varphi^p(\tau)\sum_{|k|\ge n}  |A_k(f)|^p=\varphi^p(\tau)E_{\lambda_n}^p(f)_{_{\scriptstyle p}},
 \]
 and due to the parity and non-decreasing function $\varphi$ on the interval $[0,\tau]$, for  $|h|\le \tau /\lambda_n$
 \[
 \sum_{|k|< n} \varphi^p(\lambda_kh) |A_k(f)|^p\le \sum_{\nu=1}^{n-1}  \varphi^p\Big(\frac{ \tau\lambda_\nu}{\lambda_n}\Big)
 (|A_{-\nu}(f)|^p+|A_\nu(f)|^p).
 \]
Therefore, in view of the monotonicity of the sequence of Fourier exponents $\{\lambda_k\}_{k\in {\mathbb Z}}$, we obtain
 \begin{equation}\label{Proof_Th2_2}
   \|\Delta_h^\varphi f\|_{_{\scriptstyle p}}^p\le
 \varphi^p(\tau)E_{\lambda_n}^p(f)_{_{\scriptstyle p}}+\sum_{\nu=1}^{n-1}  \varphi^p\Big(\frac{ \tau\lambda_\nu}{\lambda_n}\Big)
 H_\nu^p(f),
   \end{equation}
where $H_\nu^p(f)=(|A_{-\nu}(f)|^p+|A_\nu(f)|^p)$. Further, we use the following assertion from  \cite{Stepanets_Serdyuk_2002}:

 \begin{lemma}(\cite{Stepanets_Serdyuk_2002})
       \label{Lemma_31}
       Assume that the numerical series $ \sum _{\nu =1}^{\infty} c_{\nu } $ is convergent.
       Then for any sequence $ \beta _{\nu },$ $\nu \in {\mathbb N},$ the
       following equality holds for all positive integer  ${N_1}$  and ${N_2},$ ${N_1}\le {N_2}$:
        \begin{equation}\label{Proof_Th2_3}
        \sum _{\nu ={N_1}}^{N_2} \beta _{\nu }c_{\nu }= \beta _a\sum _{\nu ={N_1}}^{\infty }c_{\nu }+\sum _{\nu ={N_1}+1}^{N_2}( \beta _{\nu } - \beta _{\nu -1})\sum _{i=\nu }^{\infty }c_i- \beta _{N_2}\sum _{\nu ={N_2}+1}^{\infty }c_{\nu }.
        \end{equation}
 \end{lemma}

Setting  ${N_1}=1$, ${N_2}=n-1$,
 $ \beta _{\nu }=\varphi^p (\frac{ \tau\lambda_\nu}{\lambda_n} )$ and
 $c_{\nu }=H_\nu^p(f)$ in (\ref{Proof_Th2_3}), taking into account (\ref{Best_Approx}), we get
 \[
 \sum_{\nu=1}^{n-1}  \varphi^p\Big(\frac{ \tau\lambda_\nu}{\lambda_n}\Big)
 H_\nu^p(f)=\varphi^p\Big(\frac{ \tau\lambda_1}{\lambda_n}\Big)
 \sum_{\nu=1}^{\infty}   H_\nu^p(f)
 \]
 \[
   +\sum_{\nu=2}^{n-1}  \Big(\varphi^p\Big(\frac{\tau\lambda_\nu }{\lambda_n}\Big)-
    \varphi^p\Big(\frac{ \tau\lambda_{\nu-1}}{\lambda_n}\Big)\Big)\sum_{i=\nu}^\infty  H_i^p(f)-\varphi^p\Big(\frac{ \tau\lambda_{n-1}}{\lambda_n}\Big)
     \sum_{\nu=n}^{\infty}   H_\nu^p(f)
 \]
 \begin{equation}\label{Proof_Th2_4}
         = \sum_{\nu=1}^{n-1}  \Big(\varphi^p\Big(\frac{\tau\lambda_\nu }{\lambda_n}\Big)-
          \varphi^p\Big(\frac{ \tau\lambda_{\nu-1}}{\lambda_n}\Big)\Big) E_{\lambda_\nu}^p(f)_{_{\scriptstyle p}}
          -\varphi^p\Big(\frac{ \tau\lambda_{n-1}}{\lambda_n}\Big) E_{\lambda_n}^p(f)_{_{\scriptstyle p}}.
   \end{equation}
By virtue of (\ref{Proof_Th2_2}) and (\ref{Proof_Th2_4}), we get
 \[
    \|\Delta_h^\varphi f\|_{_{\scriptstyle p}}^p\le\sum_{\nu=1}^{n-1}  \Big(\varphi^p\Big(\frac{\tau\lambda_\nu }{\lambda_n}\Big)-
          \varphi^p\Big(\frac{ \tau\lambda_{\nu-1}}{\lambda_n}\Big)\Big) E_{\lambda_\nu}^p(f)_{_{\scriptstyle p}}
          -\varphi^p\Big(\frac{ \tau\lambda_{n-1}}{\lambda_n}\Big) E_{\lambda_n}^p(f)_{_{\scriptstyle p}}
 \]
\[
     + \varphi^p(\tau)E_{\lambda_n}^p(f)_{_{\scriptstyle p}}=\sum_{\nu=1}^{n }  \Big(\varphi^p\Big(\frac{\tau\lambda_\nu }{\lambda_n}\Big)-
          \varphi^p\Big(\frac{ \tau\lambda_{\nu-1}}{\lambda_n}\Big)\Big) E_{\lambda_\nu}^p(f)_{_{\scriptstyle p}},
\]
which yields (\ref{S_M.12}).
\end{proof}

Consider the case    $\varphi(t)=\varphi_\alpha  (t)=2^\alpha   |\sin (t/2) |^\alpha  $, $\alpha  >0$.
In this case, the function  $\varphi$ satisfies the conditions of Theorem
 \ref{Inverse_Theorem} with   $\tau=\pi$.  If $r=\alpha p \ge 1$, then using the inequality
$x^r  -y^r   \le r   x^{r  -1}(x-y),$ $x>0, y>0$ (see, for example, \cite[Ch.~1]{Hardy_Littlewood_Polya_1934}),
and ordinary trigonometric transformations for $\nu=1,2,\ldots,n,$  we get
\[
  \varphi^p\Big(\frac {\tau \lambda_\nu}{\lambda_n}\Big)-
  \varphi^p\Big(\frac {\tau \lambda_{\nu-1}}{\lambda_n}\Big)=
  2^{\alpha p}   \Big(\Big|\sin  \frac {\pi \lambda_\nu}{\lambda_n}  \Big|^{\alpha p}  -
  \Big|\sin  \frac {\pi \lambda_{\nu-1}}{\lambda_n}  \Big|^{\alpha p}  \Big)
\]
\[
    \le 2^{\alpha p}   {\alpha p}   |\sin  \frac {\pi \lambda_\nu}{\lambda_n}  \Big|^{{\alpha p}  -1}
    \Big|\sin  \frac {\pi \lambda_\nu}{\lambda_n}  -
  \sin  \frac {\pi \lambda_{\nu-1}}{\lambda_n}  \Big|
  \le {\alpha p}   \Big(\frac{2\pi }{\lambda_n}\Big)^{\alpha p}   \lambda_\nu^{{\alpha p}  -1}(\lambda_{\nu}-\lambda_{\nu-1}).
\]
If  $0<r  <1$, then the similar estimate can be obtained using the inequality
 $x^r  -y^r   \le r   y^{r  -1}(x-y)$, which holds for all   $x>0$ and $y>0$
 \cite[Ch.~1]{Hardy_Littlewood_Polya_1934}.

 \begin{corollary}
       \label{Corollary 21}
       Suppose  that  the function $ f\in  B{\mathcal S}^p$,    $1\le p<\infty$,  has the Fourier series of the form
       (\ref{Fourier_Series}). Then for any   $n\in {\mathbb N}$  and  $\alpha>0$,
       \begin{equation}\label{Inverse_Inequality}
       \omega _\alpha^p \Big(f, \frac{\pi}{\lambda_n}\Big)\le      \alpha p \Big(\frac{2\pi }{\lambda_n}\Big)^
       {\alpha p}
       \sum _{\nu =1}^{n}  \lambda_\nu^{\alpha p-1}(\lambda_{\nu}-\lambda_{\nu-1}) E_{\lambda_\nu}^p (f).
       \end{equation}
       If, in addition, the Fourier exponents  $\lambda_\nu$, $\nu\in {\mathbb N}$, satisfy the condition
       \begin{equation}\label{Lambda_Cond}
        \lambda_{\nu+1}-\lambda_{\nu} \le C ,\quad \nu=1,2,\ldots,
       \end{equation}
      with an absolute  constant $C>0$, then
       \begin{equation}\label{Inverse_Inequality_for_using}
       \omega _\alpha^p \Big(f, \frac{\pi}{\lambda_n}\Big)\le      \frac{\alpha p (2\pi)^{\alpha p}}{\lambda_n^{\alpha p}} C
       \sum _{\nu =1}^{n}  \lambda_\nu^{{\alpha p}-1}  E_{\lambda_\nu}^p (f).
       \end{equation}
\end{corollary}

\section{Constructive characteristics of the classes of functions defined by the generalized moduli of smoothness}

Let  $\omega$ be the function (majorant) given on  $[0,1]$. For a fixed $\alpha>0$, we set
\begin{equation} \label{omega-class}
    B{\mathcal S}^{p} H^{\omega}_{\alpha} =
    \Big\{f\in B{\mathcal S}^{p} :  \quad \omega_\alpha(f, \delta)_{_{\scriptstyle p}}=
    {\mathcal O}  (\omega(\delta)),\quad  \delta\to 0+\Big\}.
\end{equation}
Further, we consider the majorants   $\omega(\delta)$, $\delta\in [0,1]$, which satisfy the following conditions 1)--4): \noindent  {1)} $\omega(\delta)$ is continuous on $[0,1]$;\
 {  2)} $\omega(\delta)\uparrow$;\   {  3)}
$\omega(\delta)\not=0$ for $\delta\in (0,1]$;\   {  4)}~$\omega(\delta)\to 0$  for $\delta\to 0$; as well as the condition
\begin{equation} \label{B_alpha}
\quad \sum_{v=1}^n \lambda_v^{s-1}\omega\Big({1\over \lambda_v}\Big) =
{\mathcal O}  \Big[\lambda_n^s \omega \Big( {1\over \lambda_n}\Big)\Big].
\end{equation}
where $s>0$, and $\lambda_\nu$, $\nu\in {\mathbb N}$, is a increasing sequence of positive numbers.
In the case where $\lambda_\nu= \nu$, the condition $(\ref{B_alpha})$ is the known Bari condition
$({\mathscr B}_s)$ (see, e.g.
\cite{Bari_Stechkin_1956}).

 \begin{theorem}\label{Theorem 6.1}  Assume that the function $f\in B{\mathcal S}^{p}$,  $1\le p<\infty$,
  has the Fourier series of the form  (\ref{Fourier_Series}),
   $\alpha>0$ and the majorant  $\omega $ satisfies the conditions
       $1)$--\,$4)$.

       i) If  $f\in B{\mathcal S}^{p}H^{\omega}_{\alpha}$, then the following relation is true:
     \begin{equation} \label{iff-theorem}
         E_{\lambda_n}(f)_{_{\scriptstyle p}}={\mathcal O} \Big[ \omega \Big({1 \over {\lambda_n}} \Big) \Big].
      \end{equation}

      ii) If the numbers $\lambda_\nu$, $\nu\in {\mathbb N}$ satisfy condition $(\ref{Lambda_Cond})$ and
      the function $\omega^p $  satisfies condition $(\ref{B_alpha})$ with $s=\alpha p$, then relation (\ref{iff-theorem})
      yields the inclusion    $f\in B{\mathcal S}^{p}H^{\omega}_{\alpha}$.
\end{theorem}

\begin{proof} Let $f \in B{\mathcal S}^{p}H^{\omega}_{\alpha}$. Then  relation   (\ref{iff-theorem})  follows from
  (\ref{omega-class}) and (\ref{(6.16)}).

 On the other hand, if
$f\in B{\mathcal S}^{p}$, the numbers $\lambda_\nu$, $\nu\in {\mathbb N}$ satisfy condition $(\ref{Lambda_Cond})$ and
      the function $\omega^p $  satisfies condition $(\ref{B_alpha})$ with $s=\alpha p$ and relation
      (\ref{iff-theorem}) holds, then by (\ref{Inverse_Inequality_for_using}), we get
\[
    \omega _\alpha^p \Big(f, \frac {1 }{\lambda_n}\Big)_{_{\scriptstyle p}}\le
     \frac{C_1 }{\lambda_n^{\alpha p}}
       \sum _{\nu =1}^{n}  \lambda_\nu^{{\alpha p}-1}  E_{\lambda_\nu}^p (f)\le
\frac{C_1 }{\lambda_n^{\alpha p}}
       \sum _{\nu =1}^{n}  \lambda_\nu^{{\alpha p}-1}   \omega^p \Big({1\over {\lambda_\nu} }\Big)=
    {\mathcal O}  \Big[\omega^p \Big( {1\over {\lambda_n}}\Big)\Big],
\]
where $C_1=\alpha p (2\pi)^{\alpha p}\cdot C$. Hence, the function    $f$ belongs to the set   $B{\mathcal S}^{p}H^{\omega}_{\alpha}$.
\end{proof}

The function $t^r$, $0<r\le \alpha,$ satisfies condition   (\ref{B_alpha}).
 Hence, denoting by $B{\mathcal S}^{p}H_{\alpha}^r$ the class $B{\mathcal S}^{p}H^{\omega}_{\alpha}$ for
  $\omega(t)=t^r$ we establish the following statement:

\begin{corollary}\label{corollary 6.1.} Let $f \in B{\mathcal S}^{p}$,  $1\le p<\infty$,  has the Fourier series of the form  (\ref{Fourier_Series}), $\alpha >0$,
$0<r\le \alpha/p$ and condition  (\ref{Lambda_Cond}) holds. The function  $f$ belongs to the set   $B{\mathcal S}^{p}H_{\alpha}^r$, iff the  following relation is true:
$$
    E_{\lambda_n}(f)_{_{\scriptstyle p}}={\mathcal O}   ({\lambda_n^{-r}} ).
$$
\end{corollary}

In the spaces  ${\mathcal S}^p$, for  classical moduli of smoothness $\omega_m$, theorems \ref{Inverse_Theorem}
та \ref{Theorem 6.1} were proved in \cite{Stepanets_Serdyuk_2002} and
\cite{Abdullayev_Ozkartepe_Savchuk_Shidlich_2019}. In the spaces
${\mathcal S}^p$, inequalities of the form (\ref{Inverse_Inequality_for_using})  were also obtained in \cite{Sterlin_1972}.
In spaces $L_p$ of  $2\pi$periodic  functions, Lebesgue summable with the $p$th degree, inequalities of the kind
as (\ref{Inverse_Inequality_for_using}) were obtained  by M.~Timan  (see, for example,  \cite[Ch.~6]{A_Timan_M1960}, \cite[Ch.~2]{M_Timan_M2009}). In the Musielak-Orlicz type spaces, inequalities of the kind as (\ref{S_M.12}) were proved in \cite{Abdullayev_Chaichenko_Shidlich_2021}.


\end{document}